\documentclass[oneside,11pt]{amsart}%
\usepackage{amsthm,amsmath,amssymb,amsbsy,bbm,mathrsfs,supertabular}
\usepackage[authoryear]{natbib}

\RequirePackage[colorlinks,citecolor=blue,urlcolor=blue]{hyperref}

\DeclareMathAlphabet{\mathonebb}{U}{bbold}{m}{n}
\newcommand{\1}{\ensuremath{\mathonebb{1}}}
\def\B{{\mathscr B}}

\def\E{{\mathbb{E}}}

\def\IL{{\mathbb{L}}}
\def\N{{\mathbb{N}}}
 
\def\R{{\mathbb{R}}}
\def\P{{\mathbb{P}}}
\def\A{{\mathscr A}}
\def\B{{{\mathscr B}}}
\def\CC{\mathscr C}
\def\DD{{\mathscr D}}

\def\EE{{\mathscr E}}
\def\FF{{\mathscr{F}}}
\def\GG{{\mathscr{G}}}

\def\PP{\mathscr P}

\def\X{{\mathscr{X}}}

\def\WW{\mathcal{W}}

\def\ZZ{{\mathcal{Z}}}

\def\gX{{\mathbf{X}}}

\def\g0{{\mathbf{0}}}

\def\eps{{\varepsilon}}

\def\<{{\langle}}
\def\>{{\rangle}}
\def\Var{{\rm Var}}

\newcommand{\eref}[1]{(\ref{#1})}
\newcommand{\pa}[1]{\left({#1}\right)}
\newcommand{\norm}[1]{\left\|{#1}\right\|}
\newcommand{\cro}[1]{\left[{#1}\right]}
\newcommand{\ab}[1]{\left|{#1}\right|}
\newcommand{\ac}[1]{\left\{{#1}\right\}}

\newtheorem{thm}{Theorem}
\newtheorem{lem}{Lemma}
\newtheorem{prop}{Proposition}
\newtheorem{cor}{Corollary}
\newtheorem{defi}{Definition}

\def\1{1\hskip-2.6pt{\rm l}}

\usepackage{color}

\begin{document}
\title{Bounding the expectation of the supremum of an empirical process over a (weak) VC-major class}

\author{Y. Baraud}
\address{Univ. Nice Sophia Antipolis, CNRS,  LJAD, UMR 7351, 06100 Nice, France.}
\email{baraud@unice.fr}
\date{\today}

\begin{abstract}
Given a bounded class of functions $\GG$ and independent random variables $X_{1},\ldots,X_{n}$, we provide an upper bound for the expectation of the supremum of the empirical process over
elements of $\GG$ having a small variance. Our bound applies in the cases where $\GG$ is a VC-subgraph or a VC-major class and it is of smaller order than those one could get by using a universal entropy bound over the whole class $\GG$. It also involves explicit constants and does not require the knowledge of the entropy of $\GG$.
\end{abstract}

\maketitle

\section{Introduction}
The control of the fluctuations of an empirical process is a central tool in statistics for establishing
 the rate of convergence over a set of parameters of some specific estimators such as 
minimum contrast ones for example.  These techniques have been used over the years in many papers among
which van de Geer~\citeyearpar{Geer90}, Birg\'e and Massart~\citeyearpar{MR1240719}, Barron, 
Birg\'e and Massart~\citeyearpar{MR1679028} and the connections between empirical process theory 
and statistics are detailed at length in the book by van der Vaart and Wellner~
\citeyearpar{MR1385671}. With the concentration of measure phenomenon and Talagrand's
Theorem~1.4~\citeyearpar{Talagrand96} relating the control of the supremum of an empirical process 
over a class of functions $\FF$ to the expectation of this supremum, the initial problem reduces to the evaluation of that expectation. This can be done under universal entropy conditions which measure the massiveness of a class $\FF$ by bounding from above and uniformly with respect to probability measures $Q$ on $\FF$ the number $N(\FF,Q,\eps)$ of $\IL_{2}(Q)$-balls of radius $\eps$ that are necessary to cover $\FF$. A ready to use inequality is given by Theorem~3.1 in Gin\'e and Koltchinski~
\citeyearpar{MR2243881}. Roughly speaking their result says 
the following. Let $\FF$ admit an envelop function $F\le 1$ (which means that $\ab{f}\le F\le1$ for all 
$f\in\FF$) and $\log N(\FF,Q,\eps)$ be not larger than $H(\norm{F}_{\IL_{2}(Q)}/\eps)$ for some nondecreasing function $H$ independent of $Q$ and satisfying some mild conditions. Then, given $n$ i.i.d. random variables $X_{1},\ldots,X_{n}$ with an arbitrary distribution $P$,
\begin{equation}\label{gine-kolt}
\E\cro{Z(\FF)}\le C(H)\cro{\sigma \sqrt{nH\pa{2\sigma^{-1}\norm{F}_{\IL_{2}(P)}}}+H\pa{2\sigma^{-1}\norm{F}_{\IL_{2}(P)}}}
\end{equation}
where 
\begin{equation}
\qquad Z(\FF)=\sup_{f\in\FF}\ab{\sum_{i=1}^{n}\left(\strut f(X_{i})-\E\cro{f(X_{i})}\right)},
\label{Eq-ZF}
\end{equation}
$C(H)$ is a positive number depending on $H$, and $\sigma\in (0,1]$ satisfies 
$\sup_{f\in\FF}\Var(f(X_{1}))\le \sigma^{2}$.

However, computing the universal entropy of a class of functions $\FF$ is not an easy task and inequality~\eref{gine-kolt}  might not be so easy to use in general. For illustration, let us consider the case of $\FF=\GG\cap \B(g_{0},r)$ where $\GG$ is the set of nonincreasing functions from $[0,1]$ into itself and $\B(g_{0},r)$ the $\IL_{2}(P)$-ball centered at $g_{0}\in\GG$ with radius $r>0$. The universal entropy of $\FF$, which depends on the choice of $g_{0}$, is usually unknown. However, one may use that of $\GG$, which is of order $1/\eps$, to bound the universal entropy of $\FF\subset \GG$ from above. Taking for envelope function $F$ the constant function equal to 1, we derive from~\eref{gine-kolt} that there exists a universal constant $C>0$ such that
\begin{equation}
\E\cro{Z(\FF)}\le C\cro{\sqrt{n \sigma}+\sigma^{-1}}.
\label{Eq-1}
\end{equation}
While this inequality provides a satisfactory upper bound for $\E\cro{Z(\FF)}$ in general, Gin\'e and Koltchinski~\citeyearpar{MR2243881} (Example~3.8 p.1173) noticed that $\E\cro{Z(\FF)}$ was actually of smaller order than the right-hand side of \eref{Eq-1} when $g_{0}=0$. This phenomenon is actually easy to explain and we shall see that the function $g_{0}=0$ has in fact nothing magic: if $g_{0}$ is decreasing very fast on $[0,1]$ then it is quite easy to oscillate around $g_{0}$ and still remain nonincreasing on $[0,1]$. This implies that  $\GG\cap\B(g_{0},r)$ is actually massive around $g_{0}$. It is however impossible to oscillate  around a function $g_{0}$ which is constant without violating the monotonicity constraint. For a constant function $g_{0}$, $\GG\cap\B(g_{0},r)$ turns out to be less massive and $\E\cro{Z(\FF)}$ much smaller than that of the previous set. A general entropy bound on $\GG$ which allows to bound the entropies of  all sets 
$\GG\cap\B(g_{0},r)$ independently of $g_0$ therefore provides a pessimistic upper bound in the case of a constant function $g_0$. 

The above argument is not only valid when $\GG$ consists of monotone functions but more generally when $\GG$ is a bounded VC-major class on $\R$ for instance. For such a class, the family of all level sets  $\{g>c\}$ with $g\in\GG$ and $c\in\R$ form a VC-class of subsets of $\R$. When a function $g$ oscillates around $c$, the level set $\{g>c\}$ is a union of disjoint intervals and since the class of all unions of disjoint intervals is not VC, the elements of $\GG$ cannot oscillate arbitrarily around the constant function $g_{0}=c$.

The aim of this paper is to provide an upper bound for $\E\cro{Z(\FF)}$ when $\FF$ consists of the elements of a class $\GG$ (including the cases of VC-major and VC-subgraph classes) which satisfy some suitable control of their $\IL_{2}$-norms or variances.  The bounds we get are non-asymptotic, involve explicit numerical constants and are true as long as the random variables $X_{1},\ldots, X_{n}$ are independent but not necessarily i.i.d. They allow to improve the bounds one could obtain by using a naive upper bound on the entropy of the whole class~$\GG$. 

As already mentioned, the expectations of suprema of empirical processes play a central role in statistics and it is well known (we refer the reader to Theorem~5.52 in the book of van der Vaart~\citeyearpar{MR1652247} and to the historical references therein) that, given a sampling model indexed by a metric space $\Theta$, the rate of convergence of a minimum contrast estimator toward a parameter $\theta_{0}\in \Theta$  is governed by the expectation of the supremum of an empirical process over the elements $g_{\theta}$ of a class $\GG=\{g_{\theta},\ \theta\in \Theta\}$ lying within a small ball around $g_{\theta_{0}}$. Such connections between suprema of empirical processes and rates of convergence (or more generally risk bounds) of an estimator are not restricted to minimum contrast estimators and have also recently proved, in Baraud, Birg\'e and Sart~\citeyearpar{Baraud:2014qv}, to be an essential tool for the study of $\rho$-estimators. Under suitable assumptions on $\GG$ and because of the phenomenon we have explained above, one can expect some faster rates of convergence for these estimators toward specific parameters $\theta_{0}$. An illustration of this fact, which relies on the results of the present paper, can be found in Baraud and Birg\'e~\citeyearpar{Baraud:2015kq}. We show that the $\rho$-estimator built
on a class $\FF$ of densities satisfying some shape constraints achieves a rate of convergence toward some specific elements of $\FF$ which may be much faster than the minimax rate over the whole class. This phenomenon is actually not specific to $\rho$-estimators and was already observed for the Grenander estimator of a monotone density which converges at parametric rate when the target density is piecewise constant, as noticed by Birg\'e~\citeyearpar{MR1026298}, although the minimax rate over the whole set is of order $n^{-1/3}$. 

Our paper is organised as follows. The main definitions, including those of VC-classes, VC-major and weak VC-major classes, as well as some basic properties relative to these classes are given in Section~\ref{ssect-def}. The main results are presented in Section~\ref{ssect-main}. The proof of our main theorems, namely Theorems~\ref{casind} and~\ref{casind2}, are postponed to Section~\ref{sect-pm}. We also establish there upper bounds for $\E[Z(\FF)]$ in the special case where $\FF$ consists of indicator functions indexed by a class of sets $\CC$ since these bounds may be of independent interest. When $\CC$ is VC and the $X_{i}$ i.i.d., these bounds are compared to those provided by Boucheron {\it et al.}~\citeyearpar{MR3185193}. Finally Section~\ref{sect-pf} gathers the proofs of our propositions and that of Corollary~\ref{VC-M} which is specific to the case of $\FF$ being a VC-major class and $X_{1},\ldots,X_{n}$ i.i.d.

In the sequel, we shall use the following conventions and notations. The word {\it countable} will always mean {finite or countable} and, given a set $A$, $|A|$ and $\PP(A)$ will respectively denote the cardinality of $A$ and the class of all its subsets. Given two numbers $a,b$, $a\vee b$ and $a\wedge b$ mean $\max\{a,b\}$ and $\min\{a,b\}$ respectively. By convention, $\sum_{\varnothing}=0$.

\section{The setting and the main result}\label{sect-main}
Throughout the paper, $X_{1},\ldots,X_{n}$ are independent 
random variables defined on a probability space $(\Omega,\WW,\P)$ with values in a measurable space $(\X,\A)$, $\FF$ is a class of real-valued measurable functions on $(\X,\A)$ and $\eps_{1},\ldots,
\eps_{n}$ are i.i.d.\ Rademacher random variables (which means that $\eps_i$ takes the values $\pm 1$ with probability $1/2$) independent of the $X_i$. We recall that $Z(\FF)$ is defined by \eref{Eq-ZF} and set 
\[
\overline Z(\FF)=\sup_{f\in\FF}\ab{\sum_{i=1}^{n}\eps_{i}f(X_{i})}.
\]
In order to avoid measurability issues, 
$\E\cro{Z(\FF)}$ and $\E\cro{\overline Z(\FF)}$ mean $\sup_{\FF'}\E\cro{Z(\FF')}$ and 
$\sup_{\FF'}\E\cro{\overline Z(\FF')}$, respectively, where the suprema run among all countable subsets 
$\FF'$ of $\FF$. The relevance of the random variable $\overline Z(\FF)$ is due to the following classical symmetrization argument (see van der Vaart and Wellner~\citeyearpar{MR1385671}, Lemma~2.3.6) : 
%
\begin{lem}\label{L-Sym}
For all $a_{1},\ldots,a_{n}\in\R$,
\begin{equation}\label{eq-sym}
\E\cro{\sup_{f\in\FF}\ab{\sum_{i=1}^{n}\pa{f(X_{i})-\E\cro{f(X_{i})}}}}\le 2\E\cro{\sup_{f\in\FF}\ab{\sum_{i=1}^{n}\eps_{i}\pa{f(X_{i})-a_{i}}}}
\end{equation}
In particular, 
\begin{equation}\label{eq-sym1}
\E\cro{Z(\FF)}\le2\E\cro{\overline Z(\FF)}.
\end{equation}
\end{lem}
For the sake of completeness, we provide a proof in Section~\ref{sect-pm} below.

\subsection{Basic definitions and properties}\label{ssect-def}
We recall the following.
\begin{defi}
A class $\CC$ of subsets of some set $\ZZ$ is said to shatter a finite subset $Z$ of $\ZZ$ if 
$\{C\cap Z, C\in\CC\}=\PP(Z)$ or, equivalently, $|\{C\cap Z, C\in\CC\}|=2^{|Z|}$.
A non-empty class $\CC$ of subsets of $\ZZ$ is a VC-class if there exists an integer $k\in\N$ such that $\CC$ cannot shatter any subset of $\ZZ$ with cardinality larger than $k$. The dimension $d\in\N$ of $\CC$ is then the smallest of these integers $k$.
\end{defi}
Of special interest is the class $\CC$ of all intervals of $\R$ which is VC with dimension~2: for $Z=\{0,1\}$,  $\{C\cap Z, C\in\CC\}=\PP(Z)$ and whatever $Z'=\{x_{1},x_{2},x_{3}\}$ with $x_{1}<x_{2}<x_{3}$, $\{x_{1},x_{3}\}\not\in \{C\cap Z', C\in\CC\}$.

We extend this definition from classes of sets to classes of functions in the following way. 
\begin{defi}
Let $\FF$ be a non-empty class of functions on a set $\X$. We shall say that $\FF$ is weak VC-major with dimension $d\in\N$ if $d$ is the smallest integer $k\in\N$ such that, for all $u\in\R$, the class
\begin{equation}
\CC_{u}(\FF)=\left\{\strut\{x\in\X\mbox{ such that }\ f(x)>u\},\ f\in\FF\right\}
\label{Eq-Cu}
\end{equation}
is a VC-class of subsets of $\X$ with dimension not larger than $k$.  
\end{defi}
If $\FF$ consists of monotone functions on $(\X,\A)=(\R,\B(\R))$, $\CC_{u}(\FF)$ consists of intervals of $\R$  and $\FF$ is therefore weak VC-major with dimension not larger than 2. For the same reasons, this is also true for the class $\FF$ of nonnegative functions $f$ on $\R$ which are monotone on an interval of $\R$ (depending on $f$) and vanish elsewhere.

There exist other ways of extending the concept of a VC-class of sets to classes of functions. The two main ones encountered in the literature are the following: 
\begin{defi}
Let $\FF$ be a non-empty class of functions on a set $\X$. 
\begin{itemize}
\item The class $\FF$ {is VC-major} with dimension $d\in\N$ if  
\[
\CC(\FF)={\left\{\strut\{x\in\X\mbox{ such that }f(x)>u\},\ f\in\FF,\ u\in\R\right\}}
\]
is a VC-class of subsets of $\X$ with dimension $d$. 
\item The class $\FF$ is VC-subgraph with dimension $d$ if 
\[
\CC_{\times}(\FF)={\left\{\strut\{(x,u)\in\X\times \R\mbox{ such that }f(x)>u\},\ f\in\FF\right\}}
\]
is a VC-class of subsets of $\X\times \R$ with dimension $d$.
\end{itemize}
\end{defi}
These two notions are related to that of a weak VC-major class in the following way.
\begin{prop}\label{prop-VC}
If $\FF$ is either VC-major or VC-subgraph with dimension $d$ then $\FF$ is 
weak VC-major with dimension not larger than $d$.
\end{prop}
An alternative definition for a weak VC-major class can be obtained from the following proposition.
\begin{prop}\label{prop-def2}
{The class $\FF$ is weak VC-major with dimension $d$ if and only if  $d$ is the smallest integer $k\in\N$ such that, for all $u\in\R$, the class
\[
\CC_{u}^{+}(\FF)=\left\{\strut\{x\in\X\mbox{ such that }f(x)\ge u\},\ f\in\FF\right\}
\]
is a VC-class of subsets of $\X$ with dimension not larger than $k$.
}\end{prop}
The following permanence properties can be established for weak VC-major classes.
\begin{prop}\label{prop-TF}
Let $\FF$ be weak VC-major with dimension $d$. Then for any monotone function $F$, $F\circ \FF=
\{F\circ f,\ f\in\FF\}$ is weak VC-major with dimension not larger than $d$. In particular $\{-f,\ f\in\FF\}$ 
and $\{f\vee 0,\ f\in\FF\}$ are weak VC-major with respective dimensions not larger than $d$.
\end{prop}

\subsection{The main results}\label{ssect-main}
Let us first introduce some combinatoric quantities. For $u\in~(0,1)$, $\CC_{u}(\FF)$ defined by (\ref{Eq-Cu}) and $\gX=(X_{1},\ldots,X_{n})$ let  
\begin{equation}\label{eq-defgu}
\EE_{u}(\gX)=\{\{i,\ X_{i}\in C\},\ C\in\CC_{u}(\FF)\}\quad\mbox{and}\quad\Gamma_{u}=\E\cro{\log(2\ab{\EE_{u}(\gX)})}.
\end{equation}
Since $\CC_{u}(\FF)\neq \varnothing$ and $\EE_{u}(\gX)\subset \PP(\{1,\ldots,n\})$, $1\le \ab{\EE_{u}(\gX)}\le 2^{n}$. Hence, $\Gamma_{u}$ is well defined and satisfies $\log 2\le \Gamma_{u}\le(n+1)\log 2$ for all $u\in(0,1)$. The upper bound $(n+1)\log 2$ can be improved as follows when $\FF$ is weak VC-major with dimension $d$. For $u\in (0,1)$, the class $\CC_{u}(\FF)$ being VC with dimension not larger than $d$, a classical lemma of Sauer~\citeyearpar{MR0307902} (see also van der Vaart and Wellner~\citeyearpar{MR1385671}, Section 2.6.3 p.136) asserts that $\ab{\EE_{u}(\gX)}\le\sum_{j=0}^{d\wedge n}\binom{n}{j}$ for all $n\ge 1$, therefore $\Gamma_{u}\le \overline \Gamma_n(d)$ for all $u\in (0,1)$ with
\begin{equation}\label{eq-G}
\overline\Gamma_n(d)=\log\cro{2\sum_{j=0}^{d\wedge n}\binom{n}{j}}.
\end{equation}
Using the classical inequality $\sum_{j=0}^{k}\binom{n}{j}\le(en/k)^{k}$ for $k\le n$ 
(see Barron, Birg\'e and Massart~\citeyearpar{MR1679028}, Lemma~6), a convenient upper bound for $\overline \Gamma_n(d)$ when $d\ge 1$ is given by
\[
\overline \Gamma_n(d)\le \log 2+(d\wedge n)\log\pa{en\over d\wedge n}\le (d\wedge n)\log\pa{2en\over d\wedge n}.
\]
Since for $d\le n$, $\overline \Gamma_n(d)\ge\log\binom{n}{d}$, it is not difficult to see that 
\[
\overline \Gamma_n(d)=d\log n(1+o(1))\quad\mbox{when}\quad n\rightarrow+\infty.
\]
The following result holds.
\begin{thm}\label{casind}
If $\FF$ is a class of functions with values in $[0,1]$ and 
\begin{equation}\label{def-sigma}
\sigma=\sup_{f\in\FF}\left[\frac{1}{n}\sum_{i=1}^{n}\E\cro{f^{2}(X_{i})}\right]^{1/2},
\end{equation}
then, 
\begin{equation}\label{eq-vrai}
\E\cro{Z(\FF)}\le  2\sqrt{2n}\,\sigma\cro{{1\over \sigma}\int_{0}^{\sigma}\sqrt{\Gamma_{u}}du+\int_{\sigma}^{1}{\sqrt{\Gamma_{u}}\over u}du}+8\int_{0}^{1}\Gamma_{u}du,
\end{equation}
with $\Gamma_{u}$ defined by~\eref{eq-defgu}. In particular, if $\FF$ is weak VC-major with dimension $d$,
\begin{equation}\label{eq-vrai2}
\E\cro{Z(\FF)}\le 2\sqrt{\overline \Gamma_n(d)}\cro{\sigma\log\pa{e\over \sigma}\sqrt{2n}+4\sqrt{\overline \Gamma_n(d)}}
\end{equation}
with $\overline\Gamma_n(d)$ given by~\eref{eq-G}. 
\end{thm}
In view of analysing~\eref{eq-vrai2}, let $\GG$ be a weak VC-major class with dimension $d\ge 1$ consisting of functions with values in $[0,1]$, $\sigma\in [0,1]$ and 
\begin{equation}\label{def-F}
\FF=\GG(\sigma)=\ac{f\in\GG,\ \sum_{i=1}^{n}\E\cro{f^{2}(X_{i})}\le n\sigma^{2}}.
\end{equation}
As a subset of $\GG$, $\FF$ is weak VC-major with dimension not larger than $d$ and we may therefore apply our Theorem~\ref{casind} to bound $\E\cro{Z(\FF)}$ from above. When $n$ is large enough, the right-hand side of~\eref{eq-vrai2} is of order $\sigma\log(e/\sigma)\sqrt{nd\log n}$ for $\sigma\ge \sqrt{d/(n\log n)}$ and is equivalent to $2\sigma\log(e/\sigma)\sqrt{2nd\log n}$ when $\sigma$ is fixed and $n$ tends to infinity. In the opposite situation where $\sigma<\sqrt{d/(n\log n)}$, \eref{eq-vrai2} is of order $d\log n$.

For the sake of comparison with the results of  Gin\'e and Koltchinskii~\citeyearpar{MR2243881}, consider the case where the $X_{i}$ are i.i.d.\ with a nonatomic distribution $P$ on $[0,1]$, $\GG$ is the set of nondecreasing functions $f$ from $[0,1]$ into $[0,1]$ and $\FF=\GG(\sigma)$ is given by~\eref{def-F}. The class $\FF$ is weak VC-major with dimension $1$ because the elements of $\CC_{u}(\FF)$ are all of the form $(a,1]$ or $[a,1]$ with $a\in [0,1]$ for all $u$ and such classes of intervals cannot shatter a set of two elements $\{x_{1},x_{2}\}$ with $0\le x_{1}<x_{2}\le 1$ (the subset $\{x_{1}\}$ cannot be picked up). Besides, $\overline \Gamma_n(1)=\log(2(n+1))$ and Theorem~\ref{casind} gives 
\begin{equation}\label{eq-consthm1}
\E\cro{Z(\FF)}\le 2\sigma\log(e/\sigma)\sqrt{2n\log(2(n+1))}+8\log(2(n+1)).
\end{equation}
For $\sigma<e^{-e}$, Gin\'e and Koltchinskii~\citeyearpar{MR2243881} (Example 3.8 p.1173) obtained an upper bound for $\E\cro{Z(\FF)}$ of order 
\begin{equation}
B(n,\sigma)=\sigma\sqrt{nL(\sigma)}+L(\sigma)+\sqrt{\log n}
\quad\mbox{with}\quad L(\sigma)=\cro{\log\pa{\sigma^{-1}}}^{3/2}\log\log\pa{\sigma^{-1}}.
\label{Eq-GK}
\end{equation}
If $\sigma\ge \sqrt{\log n/n}$,  then $B(n,\sigma)\ge \sqrt{n}\sigma$ while 
$B(n,\sigma)\ge \sqrt{\log n}$ for $\sigma\le \sqrt{\log n/n}$. In any case, $B(n,\sigma)\ge 
\max\{\sqrt{n}\sigma, \sqrt{\log n}\}$, which shows that the bound (\ref{Eq-GK}) can only improve ours by some power of $\log n$. 

Gin\'e and Koltchinskii's bound is based on the fact that the class $\FF$ possesses an envelop function $F=\sup_{f\in \FF}f$ whose $\IL_{2}(P)$-norm equals $\sigma[\log(e/\sigma^{2})]^{1/2}$ and is therefore small when $\sigma$ is small.  This property is no longer satisfied for the class $\FF'=\{f(\cdot-t)\1_{[0,1]}(\cdot),\ t\in\R,\ f\in \FF\}$ for which $\sup_{f\in \FF'}f=1$. The elements of $\FF'$ also satisfy $\E[f^{2}(X_{1})]\le \sigma^{2}$ when the $X_{i}$ are uniformly distributed on $[0,1]$ for instance, however, while Gin\'e and Koltchinskii's trick fails for the class $\FF'$, our Theorem~\ref{casind} still applies: since $\FF'$ is weak-VC major with dimension not larger than 2 and $\overline \Gamma_n(2)\le 2\overline \Gamma_n(1)$, $\E\cro{Z(\FF')}$ is actually not larger than twice the right-hand side of~\eref{eq-consthm1}. 

When $\sigma^2$ is large enough compared to $\overline \Gamma_n(d)/n$, inequality~\eref{eq-vrai2} can be further improved as we shall see below. Let 
\begin{equation}\label{def-Ha}
\overline H(x)=x\sqrt{d\cro{5+\log\pa{1\over x}}}\ \ \ \mbox{for $x\in(0,1]$}\ \ \mbox{and}\ \ a=\pa{32\sqrt{\overline \Gamma_n(d)\over n}}\wedge 1.
\end{equation}
Note that $a=32\sqrt{(d\log n)/n}(1+o(1))$ when $n$ tends to infinity.
\begin{thm}\label{casind2}
If $\FF$ is a weak VC-major class with dimension not larger than $d\ge 1$, of functions with values in $[0,1]$,
\begin{equation}\label{eq-ZF2}
\E\cro{Z(\FF)}\le 2\E\cro{\overline Z(\FF)}\le 10\sqrt{n}B(\sigma)
\end{equation}
where  $\sigma$ is given by~\eref{def-sigma} and
\begin{equation}\label{def-B}
B(\sigma)=\left\{
\begin{array}{cl}
&\overline H\cro{\sigma\log\pa{1/\sigma}+\sigma}\ \ \mbox{for $\sigma\ge a$}\\
& \\
& \overline H\cro{\sigma\log\pa{1/a}+a}\ \ \mbox{for $\sigma< a$}
\end{array}
\right..
\end{equation}
\end{thm}
In both cases, we may note that 
\[
B(\sigma)\le \overline H\cro{(\sigma\vee a)\log\pa{e\over \sigma\vee a}}\ \ \mbox{for all $\sigma\in [0,1]$.}
\]

When $\FF=\GG(\sigma)$ is given by~\eref{def-F} and $n$ is large, the right-hand side of~\eref{eq-ZF2} is of order $\sigma\log^{3/2}(e/\sigma)\sqrt{nd}$
when $\sigma\ge a$ and improves~\eref{eq-vrai2} when $\log (1/\sigma)$ is small enough compared to $\log n$. When $\sigma< a$, two situations may occur. Either $\sigma\ge \sqrt{d/(n\log n)}$ and 
the right-hand sides of~\eref{eq-ZF2} and~\eref{eq-vrai2} are both of order $\sigma\log(e/\sigma)\sqrt{nd\log n}$, or $\sigma< \sqrt{d/(n\log n)}$ and the right-hand side of~\eref{eq-vrai2}, which is of order $d\log n$ improves that of~\eref{eq-ZF2} which is of order $d\log^{3/2}n$.

When the elements of $\FF$ take their values in $[-b,b]$ for some $b>0$, one should rather use the following result.
\begin{cor}\label{cas-gene}
Assume that $\FF$ is a weak VC-major class with dimension not larger than $d\ge 1$ consisting of functions with values in $[-b,b]$ for some $b>0$. Then,
\[
4^{-1}\E\cro{Z(\FF)}\le \cro{\sigma\log\pa{eb\over \sigma}\sqrt{2n\overline\Gamma_n(d)}+4b\overline \Gamma_n(d)}\wedge\cro{5\sqrt{n}bB(\sigma b^{-1})}.
\]
with $\overline \Gamma_n(d)$ given by~\eref{eq-G}, $\sigma$ by~\eref{def-sigma} and $B(\cdot)$ by~\eref{def-B}.
\end{cor}
\begin{proof}
By homogeneity, we may assume that $b=1$. Since $\FF$ is weak VC-major with dimension $d$, 
$\FF_{+}=\{f\vee 0,\ f\in\FF\}$ and $\FF_{-}=\{(-f)\vee 0,\ f\in\FF\}$ are both weak VC-major with 
dimension not larger than $d$ by Proposition~\ref{prop-TF}. The elements of $\FF_{+}$ and $\FF_{-}$ take their values in $[0,1]$ and 
\[
\max_{\epsilon\in\{-,+\}}\,\sup_{f\in\FF_{\epsilon}}\frac{1}{n}\sum_{i=1}^{n}\E\cro{f^{2}(X_{i})}\le\sigma^{2}.
\]
We may therefore bound $\E\cro{\sup_{f\in\FF_{\epsilon}}\ab{\sum_{i=1}^{n}\eps_{i}f(X_{i})}}$ from above for $\epsilon\in\{-,+\}$ by applying Theorems~\ref{casind} and~\ref{casind2}. {To conclude we use that $f=f\vee 0-(-f)\vee 0$ for all $f\in\FF$ so that }
\[
\E\cro{\sup_{f\in\FF}\ab{\sum_{i=1}^{n}\eps_{i}f(X_{i})}}\le \E\cro{\sup_{f\in\FF_{+}}\ab{\sum_{i=1}^{n}\eps_{i}f(X_{i})}}+\E\cro{\sup_{f\in\FF_{-}}\ab{\sum_{i=1}^{n}\eps_{i}f(X_{i})}}.
\]
\end{proof}

Finally, we conclude this section with the special case of i.i.d.\ $X_{i}$ and a VC-major class $\FF$. It is then possible to replace the control of the $\IL_{2}(P)$-norm of the elements of $\FF$ by a control of their variances. More precisely, the following holds. 
\begin{cor}\label{VC-M}
Let $X_{1},\ldots,X_{n}$ be i.i.d random variables, $\FF$ a VC-major class of functions with values in $[-b,b]$ and 
\[
\sigma=\sup_{f\in\FF}\sqrt{\Var[f(X_{1})]}\in(0,b].
\]
If $\FF$ is a VC-major class with dimension not larger than $d\ge 1$,
\[
\E\cro{Z(\FF)}\le \cro{2\sigma\log\pa{2eb\over \sigma}\sqrt{2n\overline \Gamma_n(d)}+16b\overline \Gamma_n(d)}\wedge \cro{20\sqrt{n}bB\left(b\over\sigma\right)}
\]
where $\overline \Gamma_n(d)$ is given by~\eref{eq-G} and $B(\cdot)$ by~\eref{def-B}.
\end{cor}
\section{Proofs of Theorem~\ref{casind} and~\ref{casind2}}\label{sect-pm}
\subsection{Proof of Lemma~\ref{L-Sym}}
{Let $(X'_1,\ldots,X'_n)$ be an independent copy of $\gX=(X_1,\ldots,X_n)$. Then
\begin{eqnarray*}
\E\cro{\sup_{f\in\FF}\ab{\sum_{i=1}^{n}\pa{\strut f(X_{i})-\E\cro{f(X_{i})}}}}&=& \E\cro{\sup_{f\in\FF}\ab{\sum_{i=1}^{n}\pa{f(X_{i})-\E\cro{\left.f(X_{i}')\right|\gX}}}}\\&=&\E\cro{\sup_{f\in\FF}
\ab{\E\left[\left.\sum_{i=1}^{n}\pa{f(X_{i})-f(X_{i}')}\right|\gX\right]}}\\
&\le& \E\cro{\sup_{f\in\FF}\ab{\sum_{i=1}^{n}\pa{f(X_{i})-f(X_{i}')}}}.
\end{eqnarray*}
By symmetry $\sup_{f\in\FF}\ab{\sum_{i=1}^{n}\pa{f(X_{i})-f(X_{i}')}}$ and $\sup_{f\in\FF}
\ab{\sum_{i=1}^{n}\eps_{i}\pa{f(X_{i})-f(X_{i}')}}$ have the same distribution. Therefore 
\begin{eqnarray*}
\E\cro{\sup_{f\in\FF}\ab{\sum_{i=1}^{n}\pa{f(X_{i})-f(X_{i}')}}}&=& \E\cro{\sup_{f\in\FF}\ab{\sum_{i=1}^{n}\eps_{i}\pa{f(X_{i})-a_i-\left[f(X_{i}')-a_i\right]}}}\\
&\le & 2\E\cro{\sup_{f\in\FF}\ab{\sum_{i=1}^{n}\eps_{i}\pa{f(X_{i})-a_{i}}}}.
\end{eqnarray*}
}
\subsection{The particular case of a class $\FF$ of indicator functions}
We start with the following elementary situation.
\begin{lem}\label{lem-1}
For a finite and non-empty subset $T$ of $\R^{n}$ and $v^{2}=\max_{t\in T}\sum_{i=1}^{n}t_{i}^{2}$,
\begin{equation}\label{cas-set}
\E\cro{\sup_{t\in T}\ab{\sum_{i=1}^{n}\eps_{i}t_{i}}}\le\sqrt{2\log(2\ab{T})v^{2}}.
\end{equation}
\end{lem}
\begin{proof}
For $\overline T=T\cup\{-t,\ t\in T\}$,
\[
\E\cro{\sup_{t\in T}\ab{\sum_{i=1}^{n}\eps_{i}t_{i}}}=\E\cro{\sup_{t\in\overline T}\sum_{i=1}^{n}\eps_{i}t_{i}}
\]
and the result follows from inequality~(6.3) in Massart~\citeyearpar{MR2319879}.

\end{proof}
Let us now prove an analogue of Theorem~\ref{casind} when $\FF$ is a family of indicator functions.
\begin{thm}\label{thm-set}
Let $\gX=(X_{1},\ldots,X_{n})$ be a random vector with independent components taking their values in the measurable space $(\X, \A)$ and let $\CC$ be a countable family of measurable subsets of $\X$. For $\FF=\{\1_{C},\ C\in\CC\}$, $\EE(\gX)=\ac{\strut\{i,\ X_{i}\in C\},\ C\in\CC}$,
\[
\sigma=\sup_{C\in\CC}\left[\frac{1}{n}\sum_{i=1}^{n}\P(X_{i}\in C)\right]^{1/2}\ \ \mbox{and}\ \ \ \Gamma=\E\cro{\log(2\ab{\EE(\gX)})}
\]
the following holds,
\[
\E\cro{Z(\FF)}\le2\E\cro{\overline Z(\FF)}\le 2\left[\sigma\sqrt{2n\Gamma}+4\Gamma\right].
\]
\end{thm}
This result is of the same flavour as the one Pascal Massart established in Massart~\citeyearpar{MR2319879} (see his Lemma~6.4). Massart's result involves an inexplicit constant, is established under the assumption that the $X_{i}$ are i.i.d.\ and for $\sigma$ satisfying an inequality while our bound is true for all $\sigma$. Nevertheless, the proof of 
our Theorem~\ref{thm-set} is essentially included in that provided by Massart for his  Lemma~6.4. We provide a proof below to assess the constants.
\begin{proof}
By the symmetrization argument (\ref{eq-sym}), 
\begin{eqnarray}
\E\cro{\sup_{C\in \CC}\sum_{i=1}^{n}\1_{C}(X_{i})}&\le& \E\cro{\sup_{C\in \CC}\sum_{i=1}^{n}
\pa{\1_{C}(X_{i})-\P(X_{i}\in C)}}+n\sigma^{2}\nonumber\\&\le& 2\E\cro{\sup_{C\in \CC}\ab{\sum_{i=1}^{n}
\eps_{i}\1_{C}(X_{i})}}+n\sigma^{2}=2\E\cro{\overline Z(\FF)}+n\sigma^{2}.\label{th1-et}
\end{eqnarray}
Let us denote by $\E_{\eps}$ the conditional expectation  given $\gX=(X_{1},\ldots,X_{n})$. Applying Lemma~\ref{lem-1} with $T=\{(1_{E}(1),\ldots,1_{E}(n)),\ E\in\EE(\gX)\}$ we get
\[
\E_{\eps}\cro{\sup_{C\in\CC}\ab{\sum_{i=1}^{n}\eps_{i}\1_{C}(X_{i})}}=\E_{\eps}
\cro{\max_{E\in\EE(\gX)}\ab{\sum_{i\in E}\eps_{i}}}\le\sqrt{2\log(2\ab{\EE(\gX)})\sup_{C\in \CC}\sum_{i=1}^{n}\1_{C}(X_{i})}.
\]
Taking expectations with respect to $\gX$ on both sides of this inequality, we derive from Cauchy-Schwarz's inequality and~\eref{th1-et} that
\[
\E\cro{\overline Z(\FF)}\le\sqrt{2\Gamma\E\cro{\sup_{C\in \CC}\sum_{i=1}^{n}\1_{C}(X_{i})}}\le \sqrt{2\Gamma\pa{2\E\cro{\overline Z(\FF)}+n\sigma^{2}}}.
\] 
Solving the last inequality with respect to $\E\cro{\overline Z(\FF)}$ leads to
\[
\E\cro{\overline Z(\FF)}\le \sqrt{2\Gamma n\sigma^{2}+(2\Gamma)^{2}}+2\Gamma\le 
\sqrt{2\Gamma n\sigma^{2}}+4\Gamma
\]
and the conclusion follows from \eref{eq-sym1}.
\end{proof}

Of particular interest is the situation when $\CC$ is VC with dimension $d$.
In this case, we derive from Sauer's lemma that, for all $n\ge 1$,
\[
\ab{\EE(\gX)}\le\sum_{j=0}^{d\wedge n}\binom{n}{j}.
\]
This shows that for a VC-class $\CC$ with dimension not larger than $d$, 
$\log(2\ab{\EE(\gX)})\le \overline \Gamma_n(d)$ where $\overline \Gamma_n(d)$ is given by~\eref{eq-G}. We immediately deduce from 
Theorem~\ref{thm-set} the following corollary.
\begin{cor}\label{cor-set}
Let $\gX=(X_{1},\ldots,X_{n})$ be a random vector with independent components taking their values in the measurable space $(\X, \A)$ and let $\CC$ be a countable family of measurable subsets of $\X$ which is VC with dimension $d$. For $\FF=\{\1_{C},\ C\in\CC\}$ 
\begin{equation}\label{eq-cor3}
\E\cro{Z(\FF)}\le 2\cro{\sigma\sqrt{2n\overline{\Gamma}_n(d)}+4\overline{\Gamma}_n(d)}
\quad\mbox{with}\quad\sigma=\sup_{C\in\CC}\left[\frac{1}{n}\sum_{i=1}^{n}\P(X_{i}\in C)\right]^{1/2}
\end{equation}
and $\overline \Gamma_n(d)$ given by~\eref{eq-G}.
\end{cor}
To analyse this bound, let us consider the situation where $\GG$ is the family of indicators $\{\1_{C},\ C\in\DD\}$ indexed by a VC-class  $\DD$ of subsets of $\X$ with dimension $d\ge 1$ and $\FF=\GG(\sigma)$ given by~\eref{def-F}. The bound we get on $\E\cro{Z(\FF)}$ writes as
\[
2\sqrt{2n}\sigma\sqrt{d\log n}(1+o(1))\ \ \mbox{when $n\to+\infty$}.
\]
It can be used to bound from above the smaller quantity
\[
E=\max\ac{\E\cro{\sup_{C\in\CC}\sum_{i=1}^{n}\pa{\1_{C}(X_{i})-\P(X_{i}\in C)}};\E\cro{\sup_{C\in\CC}\sum_{i=1}^{n}\pa{\P(X_{i}\in C)-\1_{C}(X_{i})}}}.
\]
When the $X_{i}$ are i.i.d., an alternative bound on $E$ is given in Theorem~13.7 of Boucheron {\it et al.}~\citeyearpar{MR3185193}. This bound, that we recall below, is based on the control of the universal entropy of a VC-class of sets which is due to Haussler~\citeyearpar{MR1313896}.
\begin{equation}\label{eq-BETAL}
E\le 72\sqrt{n}\sigma \sqrt{d\log\pa{4e^{2}\over \sigma}}\ \ \mbox{provided that}\ \ \sigma\ge 24\sqrt{{d\over 5n}\log\pa{4e^{2}\over \sigma}}.
\end{equation}
This constraint on $\sigma$ can be reformulated as $\sigma\ge \sigma_{n}$ where 
\[
\sigma_{n}={24\over \sqrt{10}}\sqrt{{d\log n\over n}}(1+o(1))\ \ \mbox{when $n\to+\infty$}.
\]
In the case $\sigma=\sigma_{n}$, inequality~\eref{eq-cor3} improves their bound in terms of constants at least when $n$ is large enough. However in the situation where $\sigma$ is fixed and $n$ is large, their bound improves ours by a $\sqrt{\log n}$ factor. We provide below an improvement of Boucheron {\it et al.}'s bound (and hence of~\eref{eq-cor3}) in terms of constants at least when $\sigma$ is large enough compared to $\sigma_{n}$. 

\begin{prop}\label{cor-set2}
Under the assumptions of Corollary~\ref{cor-set} and provided that the dimension of $\CC$ is not larger than $d\ge 1$,
\begin{equation}\label{eq-cor3b}
\E\cro{Z(\FF)}\le 2\E\cro{\overline Z(\FF)}\le 10\sqrt{n}\ \overline H\pa{\sigma\vee a} 
\end{equation}
where $\overline H$ and $a$ are given by~\eref{def-Ha}.
\end{prop}

\begin{proof}
Throughout this proof $d$ stands for $d\wedge n$. Given $\gX=(X_{1},\ldots,X_{n})$, let $P_{\gX}=n^{-1}\sum_{i=1}^{n}\delta_{X_{i}}$ be the empirical distribution based on the $X_{i}$  and for $\eta>0$ let $\CC_{\eta}=\CC_{\eta}(\gX)$ be a maximal $\eta$-separated subset of $\CC$ for the $\IL_{1}(P_{\gX})$-norm, that is, $\CC_{\eta}$ is a (random) subset of $\CC$ satisfying the following properties:  for all $C,C'\in\CC_{\eta}$ with $C\neq C'$, $|C\Delta C'|_{1,\gX}=\sum_{i=1}^{n}\ab{\1_{X_{i}\in C}-\1_{X_{i}\in C'}}>n\eta$ and for all $C\in \CC$, their exists $\Pi_{\eta}C\in\CC_{\eta}$ such that $|C\Delta \Pi_{\eta}C|_{1,\gX}\le n\eta$. Note that for $\eta<1/n$, we necessarily have that $|C\Delta \Pi_{\eta}C|_{1,\gX}=0$ which means that 
\begin{equation}\label{eq-R0}
\1_{C}(X_{i})=\1_{\Pi_{\eta}C}(X_{i})\ \ \mbox{for all $C\in\CC$ and $1\le i\le n$}.
\end{equation}
The proof is decomposed into three steps.

\paragraph{Step 1: an entropy bound.}
In the sequel, we provide an upper bound for the quantities $\log |\CC_{\eta}|$ with $\eta>0$. We first note that given two distinct sets $C,C'\in \CC_{\eta}$, $|C\Delta C'|_{1,\gX}>n\eta>0$, hence 
\[
C\cap\{X_{1},\ldots,X_{n}\}\neq C'\cap\{X_{1},\ldots,X_{n}\},
\]
and since the number of such subsets of $\{X_{1},\ldots,X_{n}\}$ is not larger than $\sum_{k=0}^{d}\binom{n}{k}$ by Sauer's lemma, we necessarily have
\[
\log \ab{\CC_{\eta}}\le \log \left[\sum_{j=0}^{d}\binom{n}{j}\right]=\overline \Gamma_n(d)-\log 2\ \ \mbox{for all $\eta>0$}.
\]
Since two arbitrary subsets $C,C'\in\CC$ satisfy $|C\Delta C'|_{1,\gX}\le n$, if $\eta\ge1$ one should take $\CC_{\eta}=\CC_{1}=\{C_{0}\}$ for some arbitrary $C_{0}\in\CC$ so that $\log|\CC_{\eta}|=0$ for all $\eta\ge 1$. 

When $\eta\in (0,1)$ there exists $k\in\{1,\ldots,n\}$ such that $(k-1)/n\le \eta<k/n$ and for all $C,C'\in\CC_{\eta}$, $|C\Delta C'|_{1,\gX}>k-1$, hence $|C\Delta C'|_{1,\gX}\ge k$, and it follows from Haussler~\citeyearpar{MR1313896} Theorem~1 that
\[
\log \left(\ab{\CC_{\eta}}\right)\le \log\cro{ e(d+1)\pa{2e\over \eta}^{d}}.
\]
Putting these bounds on $\log \ab{\CC_{\eta}}$ together we obtain that, for all $\eta>0$, 
$\log\ab{\CC_{\eta}}\le h(\eta)$ with
\[
h(\eta)=\ac{\cro{\log\pa{e(d+1)(2e)^{d}}+d\log{1\over \eta}}\wedge \cro{\overline \Gamma_n(d)-\log 2}}\1_{(0,1)}(\eta).
\]
Note that $h$ is a nonnegative, right-continuous and nonincreasing function which is bounded from above by $\overline \Gamma_n(d)-\log 2$ and satisfies for $d\ge 1$, $n\ge 1$ and $\eta\in (0,1)$,
\begin{equation}\label{eq-minh}
h(\eta)\ge \min\{2\log(2e),\log(n+1)\}\ge \log2.
\end{equation}

\paragraph{Step 2: preliminary calculations.} For $q=2^{5/2}e^{-6}\in (0,1)$, the function $H$ defined by
\[
H(x)=\int_{0}^{x}\sqrt{\log 2+h(u^{2})+h(q^{2}u^{2})}du\ \ \mbox{for $x>0$}
\]
is nondecreasing and concave. It is also differentiable from the right on $(0,+\infty)$ and its right-hand derivative at $x>0$ is given by 
\begin{equation}\label{eq-R5}
H'(x)=\sqrt{\log 2+h(x^{2})+h(q^{2}x^{2})}\le \sqrt{2\overline \Gamma_n(d)}.
\end{equation}
Besides, for $x\in(0,1)$ $H$ is differentiable and
\[
H'(x)\le\sqrt{ c_{d}+ 4d\log{1\over  x}}
\]
with
\[
c_{d}=\log 2+2\log\pa{e(d+1)(2e)^{d}}+2d\log(1/q)\le 16d\ \ \ \mbox{for $d\ge 1$}.
\]
In particular, we deduce from Jensen's inequality that for $x\in (0,1]$, 
\begin{eqnarray}
H(x)&\le &x\times {1\over x}\int_{0}^{x}\sqrt{c_{d}+4d\log{1\over u}}du\le x\cro{{1\over x}\int_{0}^{x}\pa{c_{d}+4d\log{1\over u}}du}^{1/2}\nonumber\\
&=& x\cro{c_{d}+4d\log{e\over x}}^{1/2}\le 2x\cro{d\log{e^{5}\over x}}^{1/2}=2\overline H(x).\label{eq-bHb}
\end{eqnarray}
Let
\[
\eta_{0}=\eta_{0}(\gX)=\sup_{C\in\CC}\cro{{1\over n}\sum_{i=1}^{n}\1_{X_{i}\in C}}\in [0,1].
\]
By the symmetrization argument (\ref{eq-sym}), 
\begin{eqnarray}
n\E\cro{\eta_{0}(\gX)}&\le& \E\cro{\sup_{C\in \CC}\sum_{i=1}^{n}
\pa{\1_{C}(X_{i})-\P(X_{i}\in C)}}+n\sigma^{2}\nonumber\\&\le& 2\E\cro{\sup_{C\in \CC}\ab{\sum_{i=1}^{n}
\eps_{i}\1_{C}(X_{i})}}+n\sigma^{2}=2\E\cro{\overline Z(\FF)}+n\sigma^{2}\label{th1-et1}.
\end{eqnarray}

\paragraph{Step 3: completion of the proof.} Let us now define for all positive integers $k$, $\eta_{k}=q^{2k}\eta_{0}$, $C_{k}=\Pi_{\eta_{k}}C$ for $C\in\CC$ and $T_{k}$ as the subset of $\R^{n}$  gathering those vectors of the form $(\1_{X_{1}\in C_{k+1}}-\1_{X_{1}\in C_{k}},\ldots,\1_{X_{n}\in C_{k+1}}-\1_{X_{n}\in C_{k}})$ as $C$ varies along $\CC$. For all $i\in\{1,\ldots,n\}$,
\[
\1_{X_{i}\in C}=\1_{X_{i}\in C_{0}}+\sum_{k=0}^{+\infty}\pa{\1_{X_{i}\in C_{k+1}}-\1_{X_{i}\in C_{k}}}
\]
where the sum is actually finite because of~\eref{eq-R0}. Hence, 
\[
\ab{\sum_{i=1}^{n}\eps_{i}\1_{X_{i}\in C}}\le \ab{\sum_{i=1}^{n}\eps_{i}\1_{X_{i}\in C_{0}}}+\sum_{k=0}^{+\infty}\ab{\sum_{i=1}^{n}\eps_{i}\pa{\1_{X_{i}\in C_{k+1}}-\1_{X_{i}\in C_{k}}}}
\]
and 
\[
\overline Z(\FF)\le \ab{\sum_{i=1}^{n}\eps_{i}\1_{X_{i}\in C_{0}}}+\sum_{k=0}^{+\infty}\sup_{t\in T_{k}}\ab{\sum_{i=1}^{n}\eps_{i}t_{i}}.
\]
Denoting by $\E_{\eps}$ the conditional expectation given $\gX$, the quantities $\E_{\eps}\cro{\ab{\sum_{i=1}^{n}\eps_{i}\1_{X_{i}\in C_{0}}}}$ and $\E_{\eps}\cro{\sup_{t\in T_{k}}\ab{\sum_{i=1}^{n}\eps_{i}t_{i}}}$ can be bounded from above by means of  Lemma~\ref{lem-1} using the facts that $\sum_{i=1}^{n}\1_{X_{i}\in C_{0}}\le n\eta_{0}$, $|T_{k}|\le |\CC_{\eta_{k}}||\CC_{\eta_{k+1}}|\le e^{h(\eta_{k})+h(q^{2}\eta_{k})}$  for all $k\ge 1$ and for all $C\in\CC$
\begin{eqnarray*}
\sum_{i=1}^{n}\pa{\1_{X_{i}\in C_{k+1}}-\1_{X_{i}\in C_{k}}}^{2}&=&n\ab{C_{k+1}\Delta C_{k}}_{1,\gX}\le n\cro{\ab{C_{k+1}\Delta C}_{1,\gX}+\ab{C_{k}\Delta C}_{1,\gX}}\\
&\le& n(1+q^{2})\eta_{k}=n{1+q^{2}\over (1-q)^{2}}(\sqrt{\eta_{k}}-\sqrt{\eta_{k+1}})^{2}.
\end{eqnarray*}
We get,
\begin{eqnarray*}
\E_{\eps}\cro{\overline Z(\FF)}&\le& \sqrt{2n}\cro{\sqrt{\eta_{0}\log 2}+{\sqrt{1+q^{2}}\over 1-q}\sum_{k=0}^{+\infty}\pa{\sqrt{\eta_{k}}-\sqrt{\eta_{k+1}}}\sqrt{\log 2+h(\eta_{k})+h(q^{2}\eta_{k})}}\nonumber\\
&\le& \sqrt{2n}\cro{\sqrt{\eta_{0}\log 2}+{\sqrt{1+q^{2}}\over 1-q}\sum_{k=0}^{+\infty}\int_{\sqrt{\eta_{k+1}}}^{\sqrt{\eta_{k}}}\sqrt{\log 2+h(u^{2})+h(q^{2}u^{2})}du}\\
&\le& \sqrt{2n}\cro{\sqrt{\eta_{0}\log 2}+{\sqrt{1+q^{2}}\over 1-q}\int_{0}^{\sqrt{\eta_{0}}}\sqrt{\log 2+h(u^{2})+h(q^{2}u^{2})}du}.
\end{eqnarray*}
%
%
Using~\eref{eq-minh}, 
\[
\sqrt{\eta_{0}\log 2}\le \sqrt{\log 2\over 3\log 3}\int_{0}^{\sqrt{\eta_{0}}}\sqrt{\log 2+h(u^{2})+h(q^{2}u^{2})}du
\] 
and hence, 
\[
\E_{\eps}\cro{\overline Z(\FF)}\le \sqrt{n}b_{q}H\cro{\sqrt{\eta_{0}(\gX)}}\ \ \mbox{with}\ \ b_{q}=\sqrt{2}\pa{{\sqrt{1+q^{2}}\over 1-q}+\sqrt{1\over 3}}<2.5.
\]
Taking the expectation with respect to $\gX$ on both sides and using Jensen's inequality yield to
\begin{equation}\label{eq-R2}
\E\cro{\overline Z(\FF)}\le\sqrt{n}b_{q}\E\cro{H\left(\sqrt{\eta_{0}(\gX)}\right)}\le \sqrt{n}b_{q}H\cro{\sqrt{\E\cro{\eta_{0}(\gX)}}}.
\end{equation}
If $\overline a=32(\overline \Gamma_n(d)/n)^{1/2}\ge 1$, $a=\overline a\wedge 1=1$ and
\begin{equation}\label{eq-cas1}
\E\cro{\overline Z(\FF)}\le \sqrt{n}b_{q}H\cro{\sqrt{\E\cro{\eta_{0}(\gX)}}}\le 2.5\sqrt{n}H(1)=2.5\sqrt{n}H(\sigma\vee 1).
\end{equation}
Otherwise $a=\overline a<1$ and let us set $G(u)=H(\sqrt{u})$ for $u>0$. The function $G$ is nondecreasing, concave, differentiable from the right on $(0,+\infty)$ and its right-hand derivative at $x>0$ is given by $G'(x)=H'(\sqrt{x})/(2\sqrt{x})$. In particular, using~\eref{th1-et1} and the fact that the graph of a concave function lies below its tangents, we obtain that 
\begin{eqnarray*}
H\cro{\sqrt{\E\cro{\eta_{0}(\gX)}}}&=& G(\E\cro{\eta_{0}(\gX)})\le G\pa{\sigma^{2}+2n^{-1}\E\cro{\overline Z(\FF)}}\\
&\le& G\pa{\sigma^{2}\vee a^{2}+2n^{-1}\E\cro{\overline Z(\FF)}}\\
&\le& G(\sigma^{2}\vee a^{2})+{2n^{-1}\E\cro{\overline Z(\FF)}}G'(\sigma^{2}\vee a^{2})\\
&=& H(\sigma\vee a)+{H'(\sigma\vee a)\over an}\E\cro{\overline Z(\FF)}\\
&\le& H(\sigma\vee a)+{H'(a)\over an}\E\cro{\overline Z(\FF)}.
\end{eqnarray*}
This inequality together with~\eref{eq-R2}, leads to 
\begin{equation}\label{eq-R10}
\E\cro{\overline Z(\FF)}\le \sqrt{n}b_{q}H(\sigma\vee a)+{b_{q}H'(a)\over a\sqrt{n}}\E\cro{\overline Z(\FF)}
\end{equation}
and, since by~\eref{eq-R5} and our choice of $\overline a$ (that is $\overline a>b_{q}\sqrt{2n^{-1}\overline \Gamma_n(d)}/(1-b_{q}/2.5)$),
\[
{b_{q}H'(a)\over a\sqrt{n}}\le {b_{q}\over a}\sqrt{2\overline \Gamma_n(d)\over n}\le 1-{b_{q}\over 2.5},
\]
we obtain  that
\begin{equation}\label{eq-cas2}
\E\cro{\overline Z(\FF)}\le 2.5\sqrt{n}H(\sigma\vee a).
\end{equation}

Putting~\eref{eq-cas1} and~\eref{eq-cas2} together and using~\eref{eq-bHb}, we obtain that in both cases
\[
\E\cro{\overline Z(\FF)}\le 2.5\sqrt{n}H(\sigma\vee a)\le 5\sqrt{n}\ \overline H(\sigma\vee a)
\]
and we conclude by~\eref{eq-sym1}.
\end{proof}

\subsection{Completion of the proofs of Theorem~\ref{casind} and~\ref{casind2}}
{We start with the proof of Theorem~\ref{casind}. In view of our convention about the definition of $\E\cro{Z(\FF)}$} we may assume with no loss of generality that $\FF$ is countable. Let us fix $u\in (0,1)$ and write for simplicity, $\CC_{u}(\FF)=\CC_{u}$. Since $\FF$ is weak VC-major with dimension not larger than $d$, $\CC_{u}$ is VC with dimension not larger than $d$ as well. Besides, $\CC_{u}$ is countable since $\FF$ is and by Markov's inequality
\[
\sup_{C\in\CC_{u}}\sum_{i=1}^{n}\P(X_{i}\in C)=\sup_{f\in\FF}\sum_{i=1}^{n}\P(f(X_{i})>u)\le \sup_{f\in\FF}\sum_{i=1}^{n}\cro{{\E\pa{f^{2}(X_{i})}\over u^{2}}\wedge 1}\le n\cro{{\sigma^{2}\over u^{2}}\wedge 1}.
\]
Applying Theorem~\ref{thm-set} to the class of sets $\CC_{u}$ leads to
\begin{equation}\label{eq-et1}
\E\cro{\sup_{C\in\CC_{u}}\ab{\sum_{i=1}^{n}\eps_{i}\1_{C}(X_{i})}}\le \pa{{\sigma\over u}\wedge 1}\sqrt{2n\Gamma_{u}}+ 4\Gamma_{u}.
\end{equation}
Since the elements $f\in\FF$ take their values in $[0,1]$,
\[
\ab{\sum_{i=1}^{n}\eps_{i}f(X_{i})}=\ab{\int_{0}^{1}\sum_{i=1}^{n}\eps_{i}\1_{f(X_{i})>u}\,du}\le \int_{0}^{1}\ab{\sum_{i=1}^{n}\eps_{i}\1_{f(X_{i})>u}}du.
\]
Moreover,
\[
\sup_{f\in\FF}\ab{\sum_{i=1}^{n}\eps_{i}\1_{f(X_{i})>u}}=\sup_{C\in\CC_{u}}\ab{\sum_{i=1}^{n}\eps_{i}\1_{C}(X_{i})}
\]
and it follows that
\begin{eqnarray*}
\sup_{f\in\FF}\ab{\sum_{i=1}^{n}\eps_{i}f(X_{i})}&\le& \int_{0}^{1}\sup_{C\in\CC_{u}}\ab{\sum_{i=1}^{n}\eps_{i}\1_{C}(X_{i})}du
\end{eqnarray*}
and taking expectations on both sides gives 
\begin{equation}\label{eq-imprt}
\E\cro{\overline{Z}(\FF)}\le \int_{0}^{1}\E\cro{\sup_{C\in\CC_{u}}\ab{\sum_{i=1}^{n}\eps_{i}\1_{C}(X_{i})}}du.
\end{equation}
Using~\eref{eq-et1}, 
\begin{eqnarray*}
\E\cro{\overline{Z}(\FF)}&\le&\int_{0}^{1}\cro{\pa{{\sigma\over u}\wedge 1}\sqrt{2n\Gamma_{u}}+ 4 \Gamma_{u}}du\\
&=&\sqrt{2n}\sigma\cro{{1\over \sigma}\int_{0}^{\sigma}\sqrt{\Gamma_{u}}du+\int_{\sigma}^{1}{\sqrt{\Gamma_{u}}\over u}}+4\int_{0}^{1}\Gamma_{u}du\\
\end{eqnarray*}
and the conclusion follows from~\eref{eq-sym1}.

The proof of Theorem~\ref{casind2} is quite similar except that we now bound the right-hand side of~\eref{eq-imprt} using Proposition~\ref{cor-set2}. Since $u\mapsto \overline H(u)$ is concave and nondecreasing on $[0,1]$, 
we get 
\begin{eqnarray*}
\E\cro{\sup_{f\in\FF}\ab{\sum_{i=1}^{n}\eps_{i}f(X_{i})}}
&\le& 5\sqrt{n}\int_{0}^{1}\overline H\cro{(u^{-1}\sigma)\wedge 1)\vee a}du\\
&\le& 5\sqrt{n}\ \overline H\cro{\int_{0}^{1}[(u^{-1}\sigma)\wedge 1)\vee a]du}\\
&=& 5\sqrt{n}\ \overline H\cro{\sigma\vee a-\sigma\log(\sigma\vee a)}
\end{eqnarray*}
which leads to the result.

\section{Additional proofs}\label{sect-pf}
\subsection{Proof of Proposition~\ref{prop-VC}}
If $\FF$ is VC-major with dimension $d$, $\CC(\FF)$ is a VC-class with dimension $d$ therefore, 
whatever $u\in\R$, its subset $\CC_{u}(\FF)$ is also a VC-class with dimension not larger than $d$. Let us now turn to the case where $\FF$ is VC-subgraph with dimension $d$. Let $u\in\R$, if $\CC_{u}$ shatters $\{x_{1},\ldots, x_{k}\}$, for any subset $E$ of $\{1,\ldots,k\}$ one can find a function $f\in\FF$, such that 
\[
E=\left\{i\in\{1,\ldots,k\}\;\mbox{such that}\;f(x_{i})>u\strut\right\}
\]
which exactly means that $\CC_{\times}(\FF)$ shatters $\{(x_{1},u),\ldots,(x_{k},u)\}$ and implies that 
$k\le d$.

\subsection{Proof of Proposition~\ref{prop-def2}}
For all $f\in\FF$ and $u\in\R$, we can write 
\[
\1_{\{f\ge u\}}(x)=\lim_{m\to +\infty}\1_{\{f> u-(1/m)\}}(x)\ \ \mbox{for all}\ x\in\X.
\]
This means that $\CC_{u}^{+}$ is the sequential closure of $\CC_{u}$ for the pointwise convergence of indicator functions. Lemma~2.6.17 $(vi)$ in van der Vaart and Wellner~\citeyearpar{MR1385671} (and its proof) asserts that $\CC_{u}^{+}(\FF)$ is a VC-class with dimension not larger than that of $\CC_{u}$. For the reciprocal, note that for 
all $f\in\FF$ and $u\in\R$, 
\[
\1_{\{f>u\}}(x)=\lim_{m\to +\infty}\1_{\{f\ge u+(1/m)\}}(x) \ \ \mbox{for all}\ x\in\X
\]
and conclude in the same way.

\subsection{Proof of Proposition~\ref{prop-TF}}
Let $u\in\R$. If $\CC_{u}(F\circ\FF)$ cannot shatter at least one point, its dimension is 0 and there is nothing to prove since $d\ge 0$. Otherwise, there exist $k\ge 1$ points $x_{1},\ldots, x_{k}$ in $\X$ and $m$ functions $f_{1},\ldots,f_{m}\in\FF$ such that the set $\left\{\strut\{F\circ f_{j}>u\},\ j=1,\ldots,m\right\}$ shatters $\{x_{1},\ldots, x_{k}\}$. In particular, there exists a point $x_{i}$ and a function $f_{j}$ such that $F\circ f_{j}(x_{i})\le u$ so that 
\[
s=\max_{i,j}\{f_{j}(x_{i})\;\mbox{ such that }\;F\circ f_{j}(x_{i})\le u\}
\]
is well-defined. Clearly, for all $i=1,\ldots,k$ and $j=1,\ldots,m$,
\[
F\circ f_{j}(x_{i})>u\quad\mbox{if and only if}\quad f_{j}(x_{i})>s
\]
and $\CC_{s}(\FF)$ therefore shatters $\{x_{1},\ldots,x_{k}\}$, which implies that $k\le d$.

\subsection{Proof of Corollary~\ref{VC-M}}
Let $\GG$ be the class of all functions $g_f$, $f\in\FF$, defined on $\X$ and with values in $[-b,b]$ given by
\[
g_{f}(x)={1\over 2}\pa{f(x)-\E\cro{f(X_{1})}\strut}.
\]
Since 
\[
\sup_{g\in\GG}\E\cro{g_{f}^{2}(X_{1})}={1\over 4}\sup_{f\in\FF}\Var(f(X_{1}))\le {\sigma^{2}\over 4},
\]
Corollary~\ref{VC-M} will follow from Corollary~\ref{cas-gene} if we can prove that  $\GG$ is weak VC-major. This is a consequence of the next lemma.  
\begin{lem}
If $\FF$ is VC-major with dimension $d$, $\GG$ is weak VC-major with dimension not larger than $d$.
\end{lem}
\begin{proof}
Let $u\in\R$ and $\{x_{1},\ldots,x_{k}\}$ be a nonempty subset of $\X$ which is shattered by $\CC_{u}(\GG)$ (if no such set exists then the dimension of $\CC_{u}(\GG)$ is 0 and there is nothing to prove). For any $E\subset \{1,\ldots,k\}$, there exists $f\in\FF$ such that 
\[
E=\left\{i\in\{1,\ldots,k\}\;\mbox{such that}\;g_f(x_{i})>u\strut\right\}=
\left\{i\in\{1,\ldots,k\}\;\mbox{such that}\;f(x_{i})>t\strut\right\}
\]
with $t=2(u+\E[f(X_{1})])$. Consequently, the class of sets $\CC(\FF)=\left\{\strut\{f>t\},\ f\in\FF, t\in\R\right\}$ shatters $\{x_{1},\ldots,x_{k}\}$ which implies that $k\le d$.
\end{proof}

\paragraph{\bf Acknowledgement} The author would like to thank Lucien Birg\'e for his numerous comments that have led to an improved version of the present paper.

\bibliographystyle{apalike}

\end{document}